\documentclass[a4paper,11pt]{amsart}
\usepackage[english]{babel}
\usepackage[utf8]{inputenc}
\usepackage{paralist,url,verbatim,anysize}
\usepackage{amscd}
\usepackage{fancyhdr}
\usepackage{mathrsfs}
\usepackage{nicefrac}
\usepackage[e]{esvect}
\usepackage{amsmath,amsthm,amssymb,amsfonts}
\usepackage{graphicx, graphics}
\usepackage[bookmarksopen=true]{hyperref}
\hypersetup{colorlinks=true, citecolor=blue}
\usepackage[usenames, dvipsnames]{color}
\usepackage[margin=2.5cm]{geometry}
\usepackage[nodisplayskipstretch]{setspace}
\usepackage{bbm}
\usepackage[font=small,labelfont=bf]{caption}
\usepackage{subcaption}

\theoremstyle{plain}
\newtheorem{theorem}{\bf Theorem}[]

\newtheorem{lemma}[theorem]{Lemma}

\newtheorem{thmnonumber}{\bf Theorem}

\theoremstyle{definition}

\newcommand{\RR}{\mathbb{R}}
\newcommand{\II}{\mathbb{I}}

\newcommand{\cm}[1]{}

\definecolor{mypink}{RGB}{215, 5, 234}

\begin{document}
\author{Karim Adiprasito}

\address{\emph{Karim Alexander Adiprasito}, Sorbonne Université and Université Paris Cité, CNRS, IMJ-PRG, F-75005 Paris, France}
\email{karim.adiprasito@imj-prg.fr}

\author{Bruno Benedetti}

\address{\emph{Bruno Benedetti}, Department of Mathematics, University of Miami, Florida} 
\email{bruno@math.miami.edu}

\title{Sparse handlebody decompositions \\and non-finiteness of~$g_3=0$}
\maketitle

\vskip6mm

\begin{abstract}

We prove 
 that a PL manifold admits a handle decomposition into handles of index $\le k$ if and only if $M$ is $k$-stacked, i.e., it admits a PL triangulation in which all $(d-k-1)$-faces are on $\partial M$. We use this to solve a problem posed in 2008 by Kalai: In any dimension higher than four, there are infinitely many homology-spheres with $g_3 =0$.
\end{abstract}

\phantom{a}
\vskip6mm

The face vector of a simplicial complex is a vector whose $i$-th entry counts the number of $i$-dimensional faces. Historically, face vectors have been of interest since Descartes' lost-and-found proof of the Euler characteristic of polytope boundaries, from around 1630. Surprisingly, most of the progress on the topic happened in a single decade: At the beginning of the Seventies, McMullen \cite{McMulleng} and Barnette \cite{Barnette} proved the Upper and the Lower Bound Theorem, respectively, for face vectors of polytopes. At the same time, McMullen proposed his $g$-Conjecture, which would provide a complete characterization of $f$-vectors of polytopes  \cite{McMulleng}. The two directions of the conjecture were eventually established in 1980 by Billera--Lee \cite{BLg} and by Stanley \cite{Stanleyg}. 

A lot of progress has been done since then in extending this result to spheres and manifold; see in particular the first author's extension to (homology) spheres of Stanley's theorem \cite{Adiprasitog}, and Murai--Nevo's extension to spheres of the Generalized Lower Bound Theorem \cite{MNg}. 
The latter result is of special relevance for the present paper: Under an algebraic condition now known to hold for all spheres, Murai--Nevo proved that in one of the extremal cases (namely, the case when primitive cohomology vanishes) triangulated spheres can be extended to triangulated disks with no interior faces of low dimension. This lead to an interesting sequence of questions, as it was not clear what manifolds could attain such algebraic condition.

We answer this question by relating it to the handlebody decomposition of the manifold. It is known that if a $d$-manifold $M$ admits a PL triangulation in which all $(d-k-1)$-faces are on $\partial M$, then the manifold admits also a handle decomposition into handles of index $\le k$. Swartz asked whether the converse would hold~\cite[p.~24]{Swa14}, and he was able to provide a positive answer in  a few special cases. The answer turns out to be positive in all cases:

\begin{thmnonumber}\label{thm:handlebody} Let $d \ge 2$ be an integer. Let $M$ be a PL $d$-manifold. The following are equivalent: 
\begin{compactenum}[ \rm (1) ]
\item $M$ admits a PL handle decomposition into handles of index $\le k$, 
\item $M$ admits a PL triangulation in which all $(d-k-1)$-faces are on $\partial M$.
\end{compactenum}
\end{thmnonumber}

As an application of Theorem \ref{thm:handlebody}, we provide a solution to the problem posed by Kalai, after an intuition by Swartz \cite[Problem 3.4]{Swa08}, of whether there are only finitely many homeomorphism types of PL manifolds with bounded $g_k$ and bounded homology. Our work provides a negative answer to this problem, already with the strongest restrictions possible:

\begin{thmnonumber} \label{thm:B} Let $d \ge 6$ be an integer. There are infinitely many homeomorphism types of PL $(d-1)$-manifolds that
\begin{compactenum}[ \rm (1) ]
\item have the same homology of  the $(d-1)$-sphere, and
\item admit a triangulation with $g_3=0$ (or equivalently: are boundary to a $d$-manifold $M$ that has no interior faces of dimension $d-3$).
\end{compactenum}
\end{thmnonumber}

The main ingredient of our proof of Theorem \ref{thm:handlebody} is the \emph{interpretation of stellar subdivisions as simplicial cobordisms}. More specifically, let $T$
be the $(d-1)$-dimensional boundary of some triangulated $d$-manifold $M$, and suppose $T'$ is obtained from $T$ via a stellar subdivision of a single face $F$. We view $T'$ as the boundary of a $d$-manifold obtained by attaching on top of $M$ a cone over the star of $F$. The Figure below illustrates this for $d=3$. 

\begin{figure}[htbp]
\begin{center}
\includegraphics[scale=0.24]{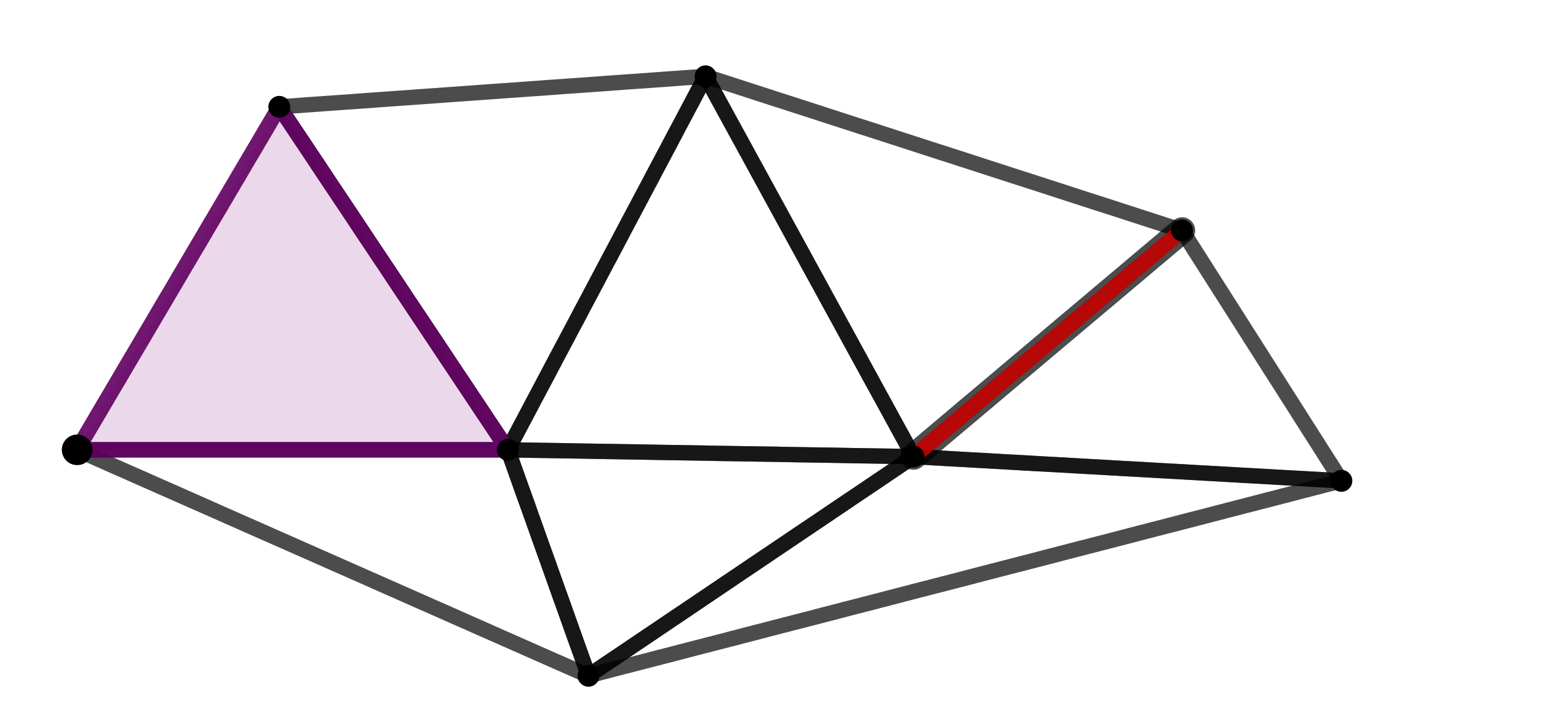} $\qquad$
\includegraphics[scale=0.24]{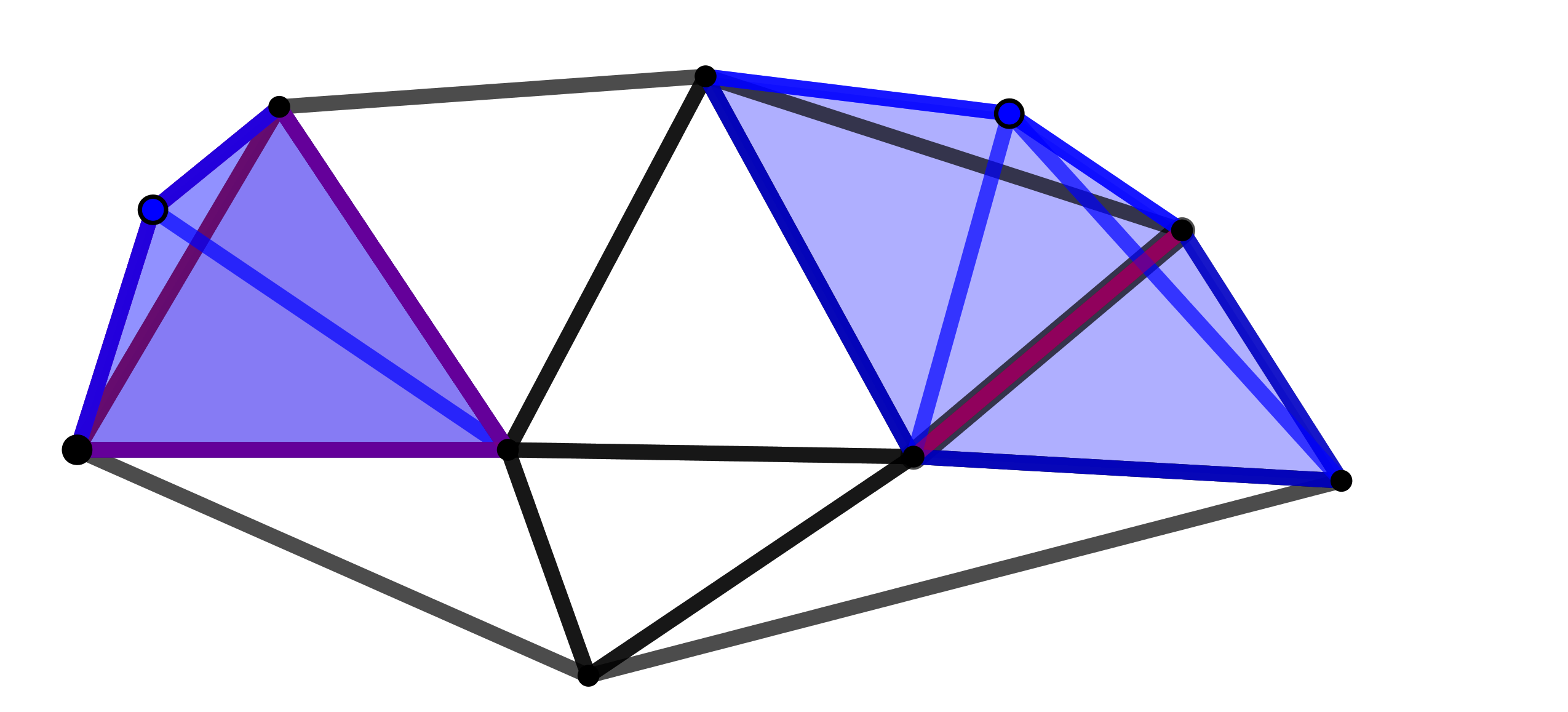}
\label{fig:ed}
\caption{To subdivide stellarly the boundary of some $3$-manifold, at a triangle (in pink) and at an edge (in red), 
we stack at the pink triangle, and we build a pyramid (consisting of two tetrahedra) on top of the star of the red edge. The boundary of the new $3$-manifold is the stellar subdivision we desired. The pink triangle, the red edge, and the two triangles containing the red edge, are now interior. No vertex of the original $3$-manifold has become interior.}
\end{center}
\end{figure}

This idea also has a long history: Smooth cobordisms are of course a standard tool in topology, and the similar idea of seeing a flip as an elementary cobordism seems folklore. See e.g. Newman \cite{Newman}. 

\medskip
\textbf{Acknowledgments} Karim Adiprasito was supported by the European Research Council under the European Unions Seventh Framework Programme ERC Grant agreement ERC StG 716424 - CASe, a DFF grant 0135-00258B, the Israel Science Foundation under ISF Grant 1050/16,  the Horizon Europe ERC Grant number: 101045750 / Project acronym: HodgeGeoComb, and the Centre National de Recherche Scientifique. Bruno Benedetti was supported by NSF Grant 1855165, Simons Award MPS-TSM-00002873, and wishes to thank Ed Swartz for many helpful conversations.
\bigskip

\enlargethispage{3mm}

\section{Proof of Theorem  \ref{thm:handlebody} }
Consider a $(d-1)$-dimensional polytopal complex $N$ (for example, the boundary of a $d$-manifold). Let $T$ be a PL triangulation of $N$. Let $C$ be a polyhedral complex mapped into $T$ along a PL map. {  By \emph{neighborhood of (the PL embedding of) $C$ inside $N$} we mean the closure of the faces of $T$ intersecting $C$}; any such neighborhood is \emph{regular} if it strongly PL  deformation retracts onto $C$. One says that $C$ is \emph{transversal} to $T$ if any $a$-face of $C$ and any $b$-face of $T$ are either disjoint, or intersect in a finite polyhedral complex of dimension $a+b- \dim T$.  {  For example, you may think of a polygonal curve $C$ inside a triangulated torus $T$; the polygonal curve is transversal to $T$ if any edge of $C$ intersects any edge of the torus in finitely many points, and in addition, all vertices of $C$ are disjoint from the $1$-skeleton of $T$.}
It is known from classical PL topology that (i) one can perturb the embedding of $C$ into $N$ so that it is transversal to $T$, and (ii) up to refining $T$ via (stellar) subdivisions, one can assume that the neighborhood of $C$ inside $T$ is regular. 
Our first goal is to strengthen (ii): namely, we want to show that one such subdivision is achievable by stellarly subdividing only faces of large dimension.

\begin{lemma} \label{lem:stellar} Let {  $1 \le k \le d$ be integers}. Let $C$ be any  $(k-1)$-complex  PL-embedded inside a $(d-1)$-complex $N$. Let $T$ be a PL triangulation of $N$ such that $C$ is transversal to $T$. There is a stellar subdivision $T'$ of $T$ 
such that 
\begin{compactenum}[ \rm (1) ]
\item the neighborhood of $C$ inside $T'$ is regular, and $C$ is transversal to {  $T'$}; 
\item $T'$ is obtained from $T$ by stellarly subdividing only faces of dimension $\ge d-k$. 
\end{compactenum}
\end{lemma}

\begin{proof} Let us start with a simple observation: If $P$ is any 
 { PL $(d-1)$-ball},
and $A, B$ are disjoint polytopes inside it so that   no facet of $\partial P$ intersects both $A$ and $B$,
then   $A$ and $B$ can be separated using stellar subdivisions at maximal faces. Indeed, a single such subdivision suffices up to PL homeomorphism.
Next, we notice that the observation above extends to the case when $A, B$ are arbitrary polytopal complexes: indeed, we only have to apply this to all pairs of faces $\alpha$ of $A$ and $\beta$ of $B$ pairwise.
{Now we are ready to prove the Lemma by induction on $d$, the cases $d=1$ and $d=2$ being trivial. By the inductive assumption, the Lemma holds for the $(d-2)$-skeleton of $T$ and the restriction of $C$ to it. For any vertex $v$ of $T$, let $B_v$ be the boundary of the star of $v$ in $T$ and let $R_C$ be the restriction of $C$ to $B_v$. Since $P_v$ has dimension $d-2$ and $R_C$ has dimension at most $k-2$ (by transversality), by inductive assumption we know that the neighborhood inside $B_v$ of the restriction  becomes regular possibly after some stellar subdivisions at faces of dimension $\ge (d-1) - (k-1)$.} This completes the proof.
\end{proof}

Our next Lemma is an interpretation of stellar subdivisions as cobordisms (cf. Figure 1):

\begin{lemma}\label{lem:stack} Let $d \ge 2$ be an integer. Let $S$ be a PL triangulation of a $d$-manifold $M$. Let $T = \partial S$. Let $T'$ be a triangulation obtained from $T$ by stellarly subdividing a single $h$-dimensional face $F$ of $T$. Then there is a PL triangulation $S'$ of $M$ such that $T' = \partial S'$, and the minimal dimension of an interior face of $S'$ is either $h$ or the minimal dimension of an interior face of $S$.
\end{lemma}

\begin{proof}
Let $v$ be a point not in $S$. Let $A$ be the star of $F$ in $T$. Let  $B$ be the cone over $A$ with apex $v$. Let $S'$ be the triangulation obtained by glueing $B$ to $S$ alongside $A$.   Note that $\partial S'= T'$.
When passing from $S$ to $S'$, the face $F$ and all faces containing it have become interior. All other faces of $T=\partial S$ are also present in $T' = \partial S'$.
\end{proof}

For the incoming two proofs we need some additional notation. We denote by $\II$ the unit interval $[0, 1] \subset \RR$. 
If  $\varphi: S^i \times \II^j \longrightarrow A$ is a PL homeomorphism and $\mathbf{m}$ is any point in $\II^j$, the image of the restriction of $\varphi$ to $S^i \times \mathbf{m}$ is called a \emph{core} of $A$. In the literature usually $\mathbf{m}$ is chosen in the relative interior of $\II^j$; but in this paper, we allow the core to touch the boundary. 

\begin{lemma} \label{lem:filling} Let $i ,j$ be nonnegative integers.
Any PL triangulation $A$ of $S^{i} \times \II^j$ in which every face of dimension $\le j-1$ is on the boundary can be extended to a PL triangulation of  the disk of dimension $i+j+1$ in which every face of dimension  { $\le j-1$} is on the boundary.
\end{lemma}

\begin{proof} For $j=0$, any PL triangulation of $S^i$ can be extended to a {  PL} triangulation of the $(i+1)$-disk, by means of coning. 
So let us focus on $j \ge 1$. 
Let $C$ be the core of~$A$. This  $C$ is not a subcomplex, but by pushing it to the boundary of $A$ we can view it as a PL-embedded $i$-sphere inside $\partial A$, which has dimension $i+j-1$. 
We may assume that $C$ {  is transversal to $\partial A$; this implies that $C$ cannot intersect faces of $\partial A$ that have dimension $\le j-2$}.  
As in Lemma~\ref{lem:stellar}, let us perform  stellar subdivisions of the faces of $\partial A$ intersected by $C$, until the neighborhood $N$ of $C$ inside the triangulation of $\partial A$ is regular. Each one of these subdivisions can be realized geometrically as in Lemma \ref{lem:stack}. Since these subdivisions are at faces of dimension $\ge j-1$, in the process described by  Lemma \ref{lem:stack} only faces of dimension $\ge j-1$ become interior, and thus the resulting triangulation of $N$ will have all faces of dimension $\le j-2$ on the boundary. But $N$ is PL homeomorphic to $S^{i} \times \II^{j-1}$. By induction, the triangulation of $N$ can be extended to a PL triangulation $D$ of the disk of dimension $i+j$ where all faces of dimension $\le j-2$ are on the boundary. Moreover, $A \cup D$ is itself an $(i+j)$-disk. If we cone over $A \cup D$, we obtained the desired $(i+j+1)$-ball.
\end{proof}

\begin{theorem} Let $M$ be a PL $d$-manifold. If $M$ has  a handle decomposition into PL handles of index $\le k$, then $M$ admits a PL triangulation with no interior faces of dimension $\le d-k-1$. 
\end{theorem}

\begin{proof} Up to rearranging the handles, one may assume that lower-index handles are attached first \cite[Section 4]{Milnorh}. The idea is to use a double induction, on the dimension $d$ and on the number of handles. Also, handles of index $0$ 
are a trivial case: A manifold with a PL handle decomposition consisting of only $0$-handles is  a disjoint union of PL $d$-balls. As such, it clearly admits a triangulation without interior $(d-1)$-faces: Namely, a disjoint union of $d$-simplices.

From now on, we assume $k \ge 1$. Suppose that we wish to attach a $k$-handle to a manifold $M$ that has a handle decomposition into PL handles of index $\le k-1$. The attachment is alongside some submanifold $A$ of $\partial M$ that is PL homeomorphic to $S^{k-1} \times  \II^{d-k}$. Let $C$ be the core of~$A$. Note that $C$ is not a subcomplex of $M$, but it is  a PL-embedded $(k-1)$-sphere inside $\partial M$. 
We may assume that $C$ is transversal to $\partial M$, which implies that $C$ may only intersect the faces of $\partial M$ that have dimension $\ge d-k$.  Now let us perform  stellar subdivisions of the faces intersected by $C$, as in Lemma \ref{lem:stellar}, to make sure that the neighborhood of $C$ inside $\partial M$ is regular. Each one of these subdivisions can be realized geometrically as in Lemma \ref{lem:stack}. Since in the process only faces of dimension $\ge d-k$ become interior, the final triangulation will have no interior faces of dimension $\le d-k-1$.
Now we are ready to perform the attachment. Let $N$ be the neighborhood of $C$ inside $\partial M$. Being regular, $N$ is PL homeomorphic to $S^{k-1} \times  \II^{d-k}$. 
 By construction, $N$ has all faces  of dimension $\le d-k-1$ on the boundary. So we can apply Lemma \ref{lem:filling}, with $i=k-1$ and $j=d-k$, and expand $N$ to a triangulation of the $d$-disk in which every face of dimension { $\le d-k-1$} is on the boundary. 
This way we have dealt  with the first $k$-handle; but the exact same argument above can be replicated also if we wish to attach a $k$-handle to a manifold that has a handle decomposition into PL handles of index $\le k$. 
 \qedhere
\end{proof}

\bigskip 

\section{Proof of Theorem \ref{thm:B}}

Recall that in abstract algebra, 
a group is called \emph{perfect} if its abelianization is trivial. 
All  non-Abelian simple groups are perfect. The converse is false, as shown by $A_5 \times A_5$: Any product of two perfect groups is perfect, but any product of two non-Abelian simple  groups is not simple. 

\begin{lemma} \label{lem:infinite}
There are infinitely many distinct balanced-presented perfect groups.
\end{lemma}

\begin{proof}
By Hausmann's trick \cite{Hau76, Hau78}, which readily produces balanced perfect groups out of arbitrary finitely presented perfect groups, we only need to construct infinitely many perfect groups that are finitely-presented. But we can easily write down infinitely many perfect groups that are finite: Take $A_5$, $\; A_5 \times A_5$, $\; A_5 \times A_5  \times A_5$,  ... and so on. 
\end{proof}

\begin{proof}[\textbf{Proof of Theorem B}]
For any $d$-manifold $M$ and for any  $k\le \frac{d}{2}$, {  Murai and Nevo \cite{MuraiNevo} showed} 
 that $\partial M$ has $g_{k+1}=0$ if and only if  $M$ can be triangulated so that all $(d-k-1)$-faces are on $\partial M$. By Theorem~\ref{thm:handlebody}, this is in turn equivalent to $M$ admitting a decomposition into handles of index~$\le k$. Now set $d \ge 6$ and $k=2$. Our goal is construct infinitely many $d$-manifolds 
using only handles of index~$\le 2$. To this end, consider any balanced presentation $\wp(G)$ of a perfect group~$G$. Then the regular neighborhood $M_{\wp(G)}$ in $\RR^d$ of the presentation complex of $\wp(G)$  is built out of handles up to index $2$ only. By a classical general-position principle, the fundamental group of the boundary of $M_{\wp(G)}$  is $G$ itself. So since the abelianization of $G$ is trivial, the boundary of $M_{\wp(G)}$ is a homology sphere. Thus any two different perfect groups $G$ and $G'$ give rise to two manifolds $M_{\wp(G)}$ and $M_{\wp(G')}$ that are not homeomorphic, because their boundaries are homology spheres that are not  homotopy equivalent.  Via Lemma~\ref{lem:infinite}, we conclude.
\end{proof}

\vskip-4mm

\end{document}